\RequirePackage{fix-cm}
\documentclass[smallextended]{svjour3}       
\smartqed  
\usepackage{graphicx}
\usepackage{amsmath}
\usepackage{amssymb}
\usepackage{mathtools}
\usepackage{algorithm}
\usepackage{algorithmic}
\newcommand{\RR}{\mathbb{R}}
\newcommand{\st}{{\rm s.t.}}
\DeclarePairedDelimiter{\norm}{\lVert}{\rVert}
\DeclareMathOperator*{\conv}{Conv}
\spnewtheorem{assumption}{Assumptions}{\bf}{\it}
\spnewtheorem{propertyy}{Property}{\bf}{\it}
\DeclareMathOperator*{\Int}{Int}

\begin{document}

\title{Primal-dual subgradient method for constrained\\ convex optimization problems
}


\author{Michael R. Metel         \and
        Akiko Takeda 
}


\institute{Michael R. Metel \at
              RIKEN Center for Advanced Intelligence Project\\
              Tokyo, Japan \\
              \email{michaelros.metel@riken.jp}           
           \and
           Akiko Takeda \at
              Department of Creative Informatics\\
              Graduate School of Information Science and Technology\\
              The University of Tokyo\\
              Tokyo, Japan\\
              RIKEN Center for Advanced Intelligence Project\\
              Tokyo, Japan\\
              \email{takeda@mist.i.u-tokyo.ac.jp}
}

\date{Received: date / Accepted: date}

\maketitle

\begin{abstract}
This paper considers a general convex constrained problem setting where functions are not assumed 
to be differentiable nor Lipschitz continuous. Our motivation is in finding a simple first-order 
method for solving a wide range of convex optimization problems with minimal requirements. We study 
the method of weighted dual averages \cite{nesterov2009} in this setting and prove that it is an 
optimal method.
\end{abstract}
\vspace{-12mm}
\section*{Declarations}
\vspace{-2mm}
\noindent Funding:
The research of the first author is supported in part by JSPS KAKENHI Grants No.19H04069. 
The research 
of the second author is supported in part by JSPS KAKENHI Grants No. 17H01699 and 19H04069.\\

\vspace{-2mm}
\noindent{Conflicts of interest/competing interests, availability of data and material, and code 
availability:}
Not applicable.

\section{Introduction}
\label{int}
In this work we are interested in constrained minimization problems,  
\begin{alignat}{6}\label{eq:1}
&&\min\limits_{x\in\RR^d}\text{ }&f(x)\\
&&\st\text{ }&f_i(x)\leq 0 \quad i=1,...,n\nonumber\\
&&&h_i(x)= 0 \quad i=1,...,p,\nonumber
\end{alignat}
where $f(x)$ and all $f_i(x)$ are convex functions and all $h_i(x)$ are affine functions from $\RR^d$ 
to $\RR$. 
Our goal is to develop an algorithm with a proven convergence rate to the constrained minimum of 
$f(x)$ without 
any further assumptions besides standard regularity conditions. In 
particular, we do not assume that $f(x)$ or the functions $\{f_i(x)\}$ are 
differentiable nor Lipschitz continuous, and we do not assume a priori that algorithm iterates are 
constrained to any bounded set. All that is required is that an optimal primal and dual 
solution exist and that strong duality holds.

Convex functions which are neither differentiable nor Lipschitz continuous arise naturally in 
optimization problems, such as the maximum of a set of quadratic functions and the unconstrained 
soft-margin SVM formulation \cite[Ch. 15]{shalev2014}. Our 
work's main motivation though is in finding a general algorithm which can be applied to a wide 
range of applications without the need for detailed function properties or problem specific 
parameter tuning.

In terms of non-asymptotic convergence guarantees for constrained non-smooth convex 
optimization problems, there are deterministic algorithms, such as the subgradient method 
\cite[Theorem 3.2.3]{nesterovintro} and its extension using mirror descent \cite{beck2010}, as 
well as algorithms for stochastic optimization settings such as the cooperative stochastic 
approximation algorithm \cite{lan2016}, which can be seen as a stochastic extension of the 
subgradient method, and the primal-dual stochastic gradient method \cite{xu2020}, which is based 
on the analysis of the augmented Lagrangian. After $K$ iterations, all of 
the algorithms discussed have a proven rate of convergence of $O(\frac{1}{\sqrt{K}})$ towards an 
optimal solution, which is the best rate achievable using a first-order method, in the sense of 
matching the lower complexity 
bound for the unconstrained version of our problem setting \cite[Section 
3.2.1]{nesterovintro}. 
All of these algorithms' convergence results rely on some combination of a compact feasible region, 
bounds on the subgradients, or bounds on the constraint functions though.

If we consider the unconstrained 
problem with $n=p=0$, recent works include \cite{grimmer2019}, which proved that the 
convergence rate of the subgradient method holds under the more relaxed assumption compared to 
Lipschitz continuity, that 
$f(x)-f(x^*)\leq D(\norm{x-x^*}_2)$ holds where $x^*$ is an optimal solution and $D(\cdot)$ is a 
non-negative non-decreasing function. The convergence rate of $O(\frac{1}{\sqrt{K}})$ for the general 
unconstrained convex optimization problem was solved earlier though in \cite{nesterov2009}, with 
the method of weighted dual averages. The iterates of the algorithm can be shown to be bounded for 
a range of convex optimization problems, including unconstrained minimization without the 
assumption of a global Lipschitz parameter. The path taken in this paper is to apply the method of 
weighted dual averages, presented as Algorithm \ref{alg:sgd}, to the general convex constrained 
problem \eqref{eq:1} and establish the same rate of convergence as previous works under our more 
relaxed 
assumptions. 

\section{Preliminaries}

We define the Lagrangian function as 
\begin{alignat}{6}
L(x,\mu,\theta):=&f(x)+\sum_{i=1}^n\mu_if_i(x)+\sum_{i=1}^p\theta_ih_i(x),\nonumber
\end{alignat}
and the dual problem as
\begin{alignat}{6}
&&\max\limits_{\substack{\mu\geq \RR^n_+\\\theta\in\RR^p}}\min\limits_{x\in\RR^d}\text{ 
}&L(x,\mu,\theta).\nonumber
\end{alignat}
The following assumptions are sufficient for \eqref{eq:1} to be a convex optimization problem with an 
optimal primal and dual solution with strong duality. 
\begin{assumption}
	\label{assump}
	\text{ }
	\begin{enumerate}
		\item $f(x)$ and $f_i(x)$ for $i=1,...,n$ are real-valued convex functions, and $h_i(x)$ 
		for $i=1,...,p$ 
		are affine functions, all over $\RR^d$.
		\item Slater's condition holds: there exists an $\hat{x}\in \RR^d$ such that $f_i(\hat{x})<0$ 
		for $i=1,...,n$ and $h_i(\hat{x})=0$ for $i=1,...,p$.
		\item There exists an optimal solution, denoted $x^*$.		
	\end{enumerate}
\end{assumption}
These assumptions are sufficient since by the finiteness of $f(x^*)$ and Slater's condition, strong 
duality holds and there exists at least one dual optimal 
solution $(\mu^*,\theta^*)$ \cite[Prop. 5.3.5]{bertsekas2009}.

Strong duality holds if and only if $(x^*,\mu^*,\theta^*)$ is a saddle point of $L(x,\mu,\theta)$  
\cite[Prop. 3.4.1]{bertsekas2009}, i.e. $\forall x\in\RR^d,\mu\geq \RR^n_+$, and $\theta\in\RR^p$,  
\begin{alignat}{6}\label{eq:2-7}
&L(x^*,\mu,\theta)\leq L(x^*,\mu^*,\theta^*) 
\leq L(x,\mu^*,\theta^*).
\end{alignat}
We will work with an unconstrained version of \eqref{eq:1}, written as 
\begin{alignat}{6}
&&\min\limits_{x\in\RR^d}\max_{\lambda\geq 0}\text{ }&F(x,\lambda):=f(x)+\lambda 
\overline{f}(x),\nonumber
\end{alignat}
where 
\begin{alignat}{6}
&&\overline{f}(x):=\max(f_1(x),f_2(x),...,f_n(x),|h_1(x)|,|h_2(x)|,...,|h_p(x)|).\nonumber
\end{alignat}
Let 
$\lambda^*=\sum_{i=1}^n\mu^*_i+\sum_{i=1}^p|\theta^*_i|$. By the fact that $\overline{f}(x^*)=0$ 
and complementary slackness, $\sum_{i=1}^n\mu^*_if_i(x^*)=0$, 
\begin{alignat}{6}
&&F(x^*,\lambda)=f(x^*)=L(x^*,\mu^*,\theta^*)\text{ }\forall \lambda,\nonumber
\end{alignat}
and
\begin{alignat}{6}
&&L(x,\mu^*,\theta^*)=&f(x)+\sum_{i=1}^n\mu^*_if_i(x)+\sum_{i=1}^p\theta^*_ih_i(x)\nonumber\\
&&\leq&f(x)+\sum_{i=1}^n\mu^*_if_i(x)+\sum_{i=1}^p|\theta^*_i||h_i(x)|\nonumber\\
&&\leq&f(x)+\lambda^*\overline{f}(x)\nonumber\\
&&=&F(x,\lambda^*),\nonumber
\end{alignat}
hence from \eqref{eq:2-7}, for all $(x,\lambda)\in\RR^{d+1}$,
\begin{alignat}{6}\label{eq:2-8}
&F(x^*,\lambda)\leq F(x,\lambda^*).
\end{alignat}
We will use the following notation for the subgradients needed of $F(x,\lambda)$,
\begin{alignat}{6}
&&g(x)\in&\partial f(x)\nonumber\\
&&g_i(x)\in&\partial f_i(x)\nonumber\\
&&\overline{g}(x)\in& \partial\overline{f}(x)=\conv(\{\partial f_i(x): 
f_i(x)=\overline{f}(x)\}\cup\{\partial|h_i(x)|:|h_i(x)|=\overline{f}(x)\})\nonumber\\
&&G_x(x,\lambda)\in&\partial_x F(x,\lambda)=\partial f(x)+\lambda\partial \overline{f}(x)\nonumber\\
&&G_{\lambda}(x,\lambda)=&\frac{\partial}{\partial \lambda}F(x,\lambda)=\overline{f}(x)\nonumber\\
&&G(x,\lambda)\in&\partial F(x,\lambda),\nonumber
\end{alignat}
where for subdifferential of $\partial \overline{f}(x)$, see for example \cite[Lemma 
3.1.10]{nesterovintro}. From a practical perspective, $\overline{g}(x)$ can be taken as an element 
$\overline{g}(x)\in \partial H(x)$, where 
$H(x)\in \{f_1(x),f_2(x),...,f_n(x),|h_1(x)|,|h_2(x)|,...,|h_p(x)|\}$ and $H(x)=\overline{f}(x)$.
Following the 
standard measure of 
convergence to a primal solution, given an optimal solution $x^*$, we 
define an algorithm's output $\bar{x}$ as an $(\epsilon_1,\epsilon_2)$-optimal solution if 
\begin{alignat}{6}
f(\bar{x})-f(x^*)\leq \epsilon_1\quad\text{and}\quad\overline{f}(\bar{x})\leq \epsilon_2.\nonumber
\end{alignat} 

\section{Weighted dual method}  

Convergence to an optimal solution is proven using the 
method of weighted dual averages of \cite{nesterov2009} presented as Algorithm \ref{alg:sgd}. When 
convenient we will use the column vector $w:=[x;\lambda]:=[x^T,\lambda^T]^T$, and the notation 
$\overline{G}_k:=\frac{[G_x(w_k);-G_{\lambda}(w_k)]}{\norm{G(w_k)}_2}$ (note that  
$\norm{\overline{G}_k}_2=1$).
\begin{algorithm}[H]
	\caption{Method of weighted dual averages}
	\label{alg:sgd}
	\begin{algorithmic}
		\STATE {\bfseries Input:} $w_0=[x_{0}\in \RR^d;\lambda_0\geq0]$; 
		$s_0=\hat{s}_0=\hat{x}_{0}=0$; $\beta_0=1$
		\FOR{$k=0,1,...,K-1$} 
		\STATE Compute  $G(w_k)\in \partial F(w_k)$ 
		\STATE $s_{k+1}=s_k+\frac{[G_x(w_k);-G_{\lambda}(w_k)]}{\norm{G(w_k)}_2}$
		\STATE $w_{k+1}=w_{0}-\frac{s_{k+1}}{\beta_{k}}$		
		\STATE $\beta_{k+1}=\beta_k+\frac{1}{\beta_k}$
		\STATE $\hat{s}_{k+1}=\hat{s}_k+\frac{1}{\norm{G(w_k)}_2}$
		\STATE $\hat{x}_{k+1}=\hat{x}_{k}+\frac{x_k}{\norm{G(w_k)}_2}$
		\ENDFOR
		\STATE Compute  $G(w_K)\in \partial F(w_K)$ 
		\STATE $\hat{s}_{K+1}=\hat{s}_K+\frac{1}{\norm{G(w_K)}_2}$
		\STATE $\hat{x}_{K+1}=\hat{x}_K+\frac{x_K}{\norm{G(w_K)}_2}$
		\RETURN $\overline{x}_{K+1}=\hat{s}_{K+1}^{-1}\hat{x}_{K+1}$
	\end{algorithmic}
\end{algorithm}
\noindent In each iteration $w_{k+1}=w_{0}-\frac{s_{k+1}}{\beta_{k}}$ is the maximizer of 
\begin{alignat}{6}\label{ufunc}
U^s_{\beta}(w):=-\langle s,w-w_0\rangle-\frac{\beta}{2}\norm{w-w_0}^2_2
\end{alignat}
for $s=s_{k+1}$ and $\beta=\beta_k$, with 
\begin{alignat}{6}\label{eq:optval}
U_{\beta_k}^{s_{k+1}}(w_{k+1})=\frac{\norm{s_{k+1}}^2_2}{2\beta_k}.
\end{alignat}
In addition, $U^s_{\beta}(w)$ is strongly concave in $w$ with parameter $\beta$,
\begin{alignat}{6}\label{strcave}
U_{\beta}^s(w)\leq U_{\beta}^s(w')+\langle \nabla U_{\beta}^s(w'),w-w'\rangle 
-\frac{\beta}{2}\norm{w-w'}^2_2.
\end{alignat}
Given that $G_{\lambda}(w_k)\geq 0$, it holds that $\lambda_{k+1}\geq \lambda_k$, with the $\lambda_k$ 
iterates always remaining feasible. The 
following property examines the case Algorithm \ref{alg:sgd} crashes due to $\norm{G(w_k)}_2=0$.
\begin{propertyy}
	\label{crash}
	If $\norm{G(w_k)}_2=0$, then $x_k$ is an optimal solution of \eqref{eq:1}.
\end{propertyy}

\begin{proof} 	
	
	If $\norm{G(w_k)}_2=0$, this implies that $G_{\lambda}(w_k)=\overline{f}(x_k)=0$ and hence $x_k$ 
	is a 
	feasible solution. From $\norm{G_{x}(w_k)}_2=0$, $0\in 
	\partial_x 
	F(x_k,\lambda_k)$, and as $F(x,\lambda_k)$ is convex in $x$, $x_k$ is a minimizer 
	of 
	$F(x,\lambda_k)$. It follows that for all $x\in\RR^d$ feasible in \eqref{eq:1},
	\begin{alignat}{6}
	f(x_k)=f(x_k)+\lambda_k\overline{f}(x_k)\leq f(x)+\lambda_k\overline{f}(x)=f(x).\nonumber
	\end{alignat}\qed
\end{proof} 
A key property of Algorithm \ref{alg:sgd} is that by redefining $G(w_k)$ appropriately, the iterates 
are bounded for quite general convex optimization problems. In particular, all that is required is 
that \eqref{conineq} in the proof below holds for the iterates to be bounded using 
\cite[Theorem 3]{nesterov2009}. For the sake of completeness we present the full proof for our 
application in Property \ref{bounded}. It is convenient to first prove a preliminary property which 
will be used in Property 
\ref{bounded} and Theorem~\ref{convergence}.

\begin{propertyy}
	\label{prelim}	
	For any iterate $\bar{k}\in\{1,2,...,K\}$ of Algorithm \ref{alg:sgd}, it holds that 
	\begin{alignat}{6}\label{constart}
\sum_{k=1}^{\bar{k}}\langle 
w_k-w_0,\overline{G}_k\rangle&\leq 
-U_{\beta_{\bar{k}}}^{s_{\bar{k}+1}}(w_{\bar{k}+1})+\frac{\beta_{\bar{k}}}{2}.
	\end{alignat}	
\end{propertyy}

\begin{proof} 	
	
From \eqref{eq:optval}, 
\begin{alignat}{6}
U_{\beta_k}^{s_{k+1}}(w_{k+1})&=\frac{\beta_{k-1}}{\beta_k}\frac{\norm{s_{k+1}}^2_2}{2\beta_{k-1}}\nonumber\\
&=\frac{\beta_{k-1}}{\beta_k}\frac{\norm{s_k+\overline{G}_k}^2_2}{2\beta_{k-1}}\nonumber\\
&=\frac{\beta_{k-1}}{\beta_k}\left(\frac{\norm{s_k}^2_2}{2\beta_{k-1}}+\frac{1}{\beta_{k-1}}\langle
s_k,\overline{G}_k\rangle+\frac{\norm{\overline{G}_k}^2_2}{2\beta_{k-1}}\right)\nonumber\\
&=\frac{\beta_{k-1}}{\beta_k}(U_{\beta_{k-1}}^{s_k}(w_k)+\frac{1}{\beta_{k-1}}\langle 
s_k,\overline{G}_k\rangle+\frac{1}{2\beta_{k-1}})\nonumber\\
&=\frac{\beta_{k-1}}{\beta_k}(U_{\beta_{k-1}}^{s_k}(w_k)+\langle 
w_0-w_k,\overline{G}_k\rangle+\frac{1}{2\beta_{k-1}}).\nonumber
\end{alignat}
Rearranging,
\begin{alignat}{6}
\langle 
w_k-w_0,\overline{G}_k\rangle 
&=U_{\beta_{k-1}}^{s_k}(w_k)-\frac{\beta_k}{\beta_{k-1}}U_{\beta_k}^{s_{k+1}}(w_{k+1})+\frac{1}{2\beta_{k-1}}\nonumber\\
&\leq U_{\beta_{k-1}}^{s_k}(w_k)-U_{\beta_k}^{s_{k+1}}(w_{k+1})+\frac{1}{2\beta_{k-1}},\nonumber
\end{alignat}	
since $\beta_k$ is increasing. Telescoping these inequalities for $k=1,..,\bar{k}$, 
\begin{alignat}{6}
\sum_{k=1}^{\bar{k}}\langle 
w_k-w_0,\overline{G}_k\rangle&\leq 
U_{\beta_{0}}^{s_1}(w_1)-U_{\beta_{\bar{k}}}^{s_{\bar{k}+1}}(w_{\bar{k}+1})+\sum_{k=1}^{\bar{k}}\frac{1}{2\beta_{k-1}}\label{kgo}\\
&=\frac{1}{2\beta_0}-U_{\beta_{\bar{k}}}^{s_{\bar{k}+1}}(w_{\bar{k}+1})+\sum_{k=0}^{\bar{k}-1}\frac{1}{2\beta_k}\nonumber\\
&=-U_{\beta_{\bar{k}}}^{s_{\bar{k}+1}}(w_{\bar{k}+1})+\frac{1}{2}(\sum_{k=0}^{\bar{k}-1}\frac{1}{\beta_k}+\beta_0),\nonumber
\end{alignat}	
where $\norm{s_1}_2=1$ was used in the first equality, and $\beta_0=1$ was used in 
the second equality. Expanding the recursion $\beta_k=\frac{1}{\beta_k-1}+\beta_{k-1}$,
\begin{alignat}{6}
\sum_{k=1}^{\bar{k}}\langle 
w_k-w_0,\overline{G}_k\rangle&\leq 
-U_{\beta_{\bar{k}}}^{s_{\bar{k}+1}}(w_{\bar{k}+1})+\frac{\beta_{\bar{k}}}{2}.\nonumber
\end{alignat}\qed	

\end{proof}

\begin{propertyy}
	\label{bounded}	
	For any iterate $\bar{k}\in\{1,2,...,K\}$ of Algorithm \ref{alg:sgd}, it holds that  	
	\begin{alignat}{6}
	\norm{w_{\bar{k}}-w^*}_2&\leq\norm{w_0-w^*}_2+1,\nonumber
	\end{alignat}
	with the inequality being strict when $w_0\neq w^*$.
\end{propertyy}

\begin{proof} 	
	
	Given the convexity of $F(x,\lambda)$ in $x$ and linearity in $\lambda$,
	\begin{alignat}{6}
	&F(x^*,\lambda_k)&\geq F(x_k,\lambda_k)+ \langle G_x(x_k,\lambda_k),x^*-x_k\rangle\label{4-1}
	\end{alignat} 	
	and
	\begin{alignat}{6}		
	&F(x_k,\lambda^*)&= F(x_k,\lambda_k)+ \langle 
	G_\lambda(x_k,\lambda_k),\lambda^*-\lambda_k\rangle.\label{4-2}
	\end{alignat}  	
	Subtracting \eqref{4-1} from \eqref{4-2} and using \eqref{eq:2-8},
	\begin{alignat}{6}\label{conineq}
	&0\leq \langle [G_x(w_k);-G_\lambda(w_k)],w_k-w^*\rangle.
	\end{alignat} 	
	
	For the case when $\bar{k}=1$,
	\begin{alignat}{6}
	\norm{w_1-w^*}^2_2=&\norm{w_0-w^*-\overline{G}_0}^2_2\nonumber\\
	=&\norm{w_0-w^*}^2_2-2\langle w_0-w^*, 
	\overline{G}_0\rangle+1\nonumber\\
	\leq&\norm{w_0-w^*}^2_2+1,\label{case1}
	\end{alignat}
	where the last line uses \eqref{conineq}. When $\bar{k}>1$ from \eqref{conineq},  
	\begin{alignat}{6}
	0\leq&\sum_{k=1}^{\bar{k}-1}\langle 
	w_k-w^*,\overline{G}_k\rangle\nonumber\\
	=&\sum_{k=1}^{\bar{k}-1}\langle 
	w_0-w^*,\overline{G}_k\rangle+\sum_{k=1}^{\bar{k}-1}\langle 
	w_k-w_0,\overline{G}_k\rangle\nonumber\\
	\leq&\sum_{k=1}^{\bar{k}-1}\langle 
	w_0-w^*,\overline{G}_k\rangle-U_{\beta_{\bar{k}-1}}^{s_{\bar{k}}}(w_{\bar{k}})+\frac{\beta_{\bar{k}-1}}{2}\nonumber\\
	=&\langle 
	w_0-w^*,s_{\bar{k}}\rangle-U_{\beta_{\bar{k}-1}}^{s_{\bar{k}}}(w_{\bar{k}})+\frac{\beta_{\bar{k}-1}}{2},\label{strcon}
	\end{alignat}	
	where the second inequality uses Property \ref{prelim}, and the second equality holds since 
	$s_{k+1}=s_k+\overline{G}_k$. Considering inequality \eqref{strcave} with $s=s_{\bar{k}}$, 
	$\beta=\beta_{\bar{k}-1}$, $w=w^*$, and $w'=w_{\bar{k}}$,
	\begin{alignat}{6}
	U_{\beta_{\bar{k}-1}}^{s_{\bar{k}}}(w^*)&\leq 
	U_{\beta_{\bar{k}-1}}^{s_{\bar{k}}}(w_{\bar{k}})+\langle \nabla 
	U_{\beta_{\bar{k}-1}}^{s_{\bar{k}}}(w_{\bar{k}}),w^*-w_{\bar{k}}\rangle 
	-\frac{\beta_{\bar{k}-1}}{2}\norm{w^*-w_{\bar{k}}}^2_2\nonumber\\
	&= 
	U_{\beta_{\bar{k}-1}}^{s_{\bar{k}}}(w_{\bar{k}})-\frac{\beta_{\bar{k}-1}}{2}\norm{w^*-w_{\bar{k}}}^2_2,\nonumber
	\end{alignat}	
	given that $w_{\bar{k}}$ is the maximum of $U_{\beta_{\bar{k}-1}}^{s_{\bar{k}}}(w)$.
	Applying this inequality in \eqref{strcon},
	\begin{alignat}{6}
	0\leq&\langle 
	w_0-w^*,s_{\bar{k}}\rangle-U_{\beta_{\bar{k}-1}}^{s_{\bar{k}}}(w^*)
	-\frac{\beta_{\bar{k}-1}}{2}\norm{w^*-w_{\bar{k}}}^2_2+\frac{\beta_{\bar{k}-1}}{2}\nonumber\\
	=&\langle 
	w_0-w^*,s_{\bar{k}}\rangle+\langle 
	s_{\bar{k}},w^*-w_0\rangle+\frac{\beta_{\bar{k}-1}}{2}\norm{w^*-w_0}^2_2
	-\frac{\beta_{\bar{k}-1}}{2}\norm{w^*-w_{\bar{k}}}^2_2+\frac{\beta_{\bar{k}-1}}{2}\nonumber\\
	=&\frac{\beta_{\bar{k}-1}}{2}\norm{w^*-w_0}^2_2
	-\frac{\beta_{\bar{k}-1}}{2}\norm{w^*-w_{\bar{k}}}^2_2+\frac{\beta_{\bar{k}-1}}{2},\nonumber
	\end{alignat}
	where the first equality uses the definition of $U_{\beta_{\bar{k}-1}}^{s_{\bar{k}}}(w^*)$ 
	\eqref{ufunc}. 
	Rearranging,
	\begin{alignat}{6}
	\norm{w^*-w_{\bar{k}}}^2_2\leq&\norm{w^*-w_0}^2_2+1.\label{normsq}
	\end{alignat}
	Now for all $\bar{k}$, from \eqref{case1} and \eqref{normsq},  
	\begin{alignat}{6}
	(\norm{w^*-w_0}_2+1)^2=&\norm{w^*-w_0}^2_2+2\norm{w^*-w_0}+1\nonumber\\
	\geq&\norm{w^*-w_{\bar{k}}}^2_2+2\norm{w^*-w_0},\nonumber
	\end{alignat}
	so that 
	\begin{alignat}{6}
	\label{bound2}
	\norm{w^*-w_0}_2+1\geq\norm{w^*-w_{\bar{k}}}_2,
	\end{alignat}
	with \eqref{bound2} being strict when $w^*\neq w_0$.\qed
\end{proof}
In order to prove the convergence result of Algorithm \ref{alg:sgd}, we require bounding 
the norm of the subgradients $G(w_k)$. 

\begin{propertyy}
	\label{gradbound}
	There exists a constant $L$ such that $\norm{g(x_k)}_2\leq L$, 
	$\norm{\overline{g}(x_k)}_2\leq L$, $\overline{f}(x_k)\leq L\norm{x_k-x^*}_2$, and 
	\begin{alignat}{6}
	\norm{G(w_k)}_2\leq&L(2\norm{w_0-w^*}_2+\lambda^*+3)\nonumber
	\end{alignat}
	for all $k$. 
\end{propertyy}

\begin{proof} 	
	Recall that $g(x)\in \partial f(x)$ and $\overline{g}(x)\in \partial \overline{f}(x)$, 
	\begin{alignat}{6}
	\norm{G(w_k)}_2=&\norm{[G_x(w_k);G_{\lambda}(w_k)]}_2\nonumber\\
	=&\norm{[g(x_k)+\lambda_k\overline{g}(x_k);\overline{f}(x_k)]}_2\nonumber\\
	\leq&\norm{g(x_k)}_2+\lambda_k\norm{\overline{g}(x_k)}_2+\overline{f}(x_k).\label{13}
	\end{alignat}	
	Property 3 ensures that the iterates of Algorithm \ref{alg:sgd} are bounded in a convex compact 
	region, 
	$w_k\in 
	D:=\{w : \norm{w-w^*}_2\leq \norm{w_0-w^*}_2+1\}$. This implies that $x_k\in 
	D_x:=\{x : \norm{x-x^*}_2\leq \norm{w_0-w^*}_2+1\}$ and $\lambda_k\in 
	D_{\lambda}:=\{\lambda : |\lambda-\lambda^*|\leq \norm{w_0-w^*}_2+1\}$. 
	It follows that there exists an $L_1\geq 0$ such that $f(x)$ is $L_1$-Lipschitz continuous 
	on $D_x$ 
	\cite[Theorem IV.3.1.2]{hiriart1996},
	\begin{alignat}{6}\label{Lip}
	|f(x)-f(x')|\leq L_1\norm{x-x'}_2,
	\end{alignat}	
	for all $x,x'\in D_x$. Assuming that $w_0\neq w^*$, $x_k\in \Int D_x$. For 
	any $x\in \Int D_x$, taking $\theta>0$ small enough such 
	that $x'=x+\theta\frac{g(x)}{\norm{g(x)}_2}\in D_x$,
	\begin{alignat}{6}
	&\quad&\langle g(x), x'-x\rangle&\leq f(x')-f(x)\nonumber\\
	\Longrightarrow&\quad&\langle g(x), x'-x\rangle&\leq L_1\norm{x'-x}_2\nonumber\\
	\Longrightarrow&\quad&\langle g(x), \theta\frac{g(x)}{\norm{g(x)}_2}\rangle&\leq 
	L_1\theta\nonumber\\
	\Longrightarrow&\quad&\norm{g(x)}_2&\leq L_1.\label{gradbound2}
	\end{alignat}    
	If $w_0=w^*$, $x_k\in\Int D_x^\delta:=\{x : \norm{x-x^*}_2\leq \delta+1\}$ for any
	$\delta>0$, and $L_1$ can be increased such that \eqref{Lip} holds over $D^{\delta}_x$ so that 
	\eqref{gradbound2} holds for all $x\in D_x$. Similarly, there exists an 
	$L_2\geq0$ such that $|\overline{f}(x)-\overline{f}(x')|\leq L_2\norm{x-x'}_2$
	and $\norm{\overline{g}(x)}_2\leq L_2$ for all 
	$x,x'\in D_x$. In addition,
	\begin{alignat}{6}
	\overline{f}(x_k)&=|\overline{f}(x_k)-\overline{f}(x^*)|\nonumber\\
	&\leq L_2\norm{x_k-x^*}_2\nonumber\\
	&\leq L_2(\norm{w_0-w^*}_2+1),\nonumber
	\end{alignat}
	and $\lambda_k\leq \norm{w_0-w^*}_2+1+\lambda^*$ from the definition of $D_{\lambda}$. Combining 
	these bounds in \eqref{13} and 
	taking 
	$L=\max(L_1,L_2)$, 
	\begin{alignat}{6}
	\norm{G(w_k)}_2\leq&\norm{g(x_k)}_2+\lambda_k\norm{\overline{g}(x_k)}_2+\overline{f}(x_k)\nonumber\\
	\leq&L_1+(\norm{w_0-w^*}_2+1+\lambda^*)L_2+L_2(\norm{w_0-w^*}_2+1)\nonumber\\
	\leq&L(2\norm{w_0-w^*}_2+\lambda^*+3).\nonumber
	\end{alignat}\qed	
\end{proof}
For all $k$ the value of $\beta_k$ can be bounded as follows using induction.
\begin{propertyy}{\cite[Lemma 3]{nesterov2009}}
	\label{betabound}
	\begin{alignat}{6}
	\beta_{k}\leq \frac{1}{1+\sqrt{3}}+\sqrt{2k+1}\nonumber
	\end{alignat}
\end{propertyy}
We can now prove a convergence rate of $O(\frac{1}{\sqrt{K}})$ to an optimal solution of problem 
\eqref{eq:1}. We will define the bound on 
$\norm{G(w_k)}_2$ from Property \eqref{gradbound} as $C:=L(2\norm{w_0-w^*}_2+\lambda^*+3)$.
\begin{theorem}
	\label{convergence}
	Running Algorithm \ref{alg:sgd} for $K$ iterations, 
	\begin{alignat}{6}
	f(\bar{x}_{K+1})-f(x^*)\leq 
	\frac{C(\norm{w_0-w^*}^2_2+1)}{2(K+1)}\left(\frac{1}{1+\sqrt{3}}+\sqrt{2K+1}\right)\nonumber
	\end{alignat}
	and
	\begin{alignat}{6}
	\overline{f}(\bar{x}_{K+1})\leq 
	\frac{C(4(\norm{w_0-w^*}_2+1)^2+1)}{2(K+1)}\left(\frac{1}{1+\sqrt{3}}+\sqrt{2K+1}\right),\nonumber
	\end{alignat}	
	where $C:=L(2\norm{w_0-w^*}_2+\lambda^*+3)$.
\end{theorem}

\begin{proof} 	
	Applying Property \ref{prelim} with $\bar{k}=K$, and recalling that 
	$w_{K+1}$ maximizes 
	$U_{\beta_K}^{s_{K+1}}(w)$ 
	\eqref{ufunc}, 
	\begin{alignat}{6}
	\frac{\beta_K}{2}\geq&\sum_{k=1}^K\langle 
	w_k-w_0,\overline{G}_k\rangle+U_{\beta_K}^{s_{K+1}}(w_{K+1})\nonumber\\
	=&\sum_{k=1}^K\langle 
	w_k-w_0,\overline{G}_k\rangle+\max\limits_{w\in\RR^{d+1}}\lbrace-\langle 
	\sum_{k=0}^K\overline{G}_k,w-w_0\rangle-\frac{\beta_K}{2}\norm{w-w_0}^2_2\rbrace\nonumber\\
	=&\max\limits_{w\in\RR^{d+1}}\lbrace 
	\sum_{k=0}^K-\langle 
	\overline{G}_k,w-w_k\rangle-\frac{\beta_K}{2}\norm{w-w_0}^2_2\rbrace.\label{prfin}
	\end{alignat}	
	Like $\overline{x}_{K+1}$, let 
	$\overline{w}_{K+1}:=\hat{s}_{K+1}^{-1}\sum_{k=0}^K\frac{w_k}{\norm{G(w_k)}_2}$
	and $\overline{\lambda}_{K+1}:=\hat{s}_{K+1}^{-1}\sum_{k=0}^K\frac{\lambda_k}{\norm{G(w_k)}_2}$.
	Multiplying both sides of \eqref{prfin} by $\hat{s}_{K+1}^{-1}$,
	\begin{alignat}{6}
	\hat{s}_{K+1}^{-1}\frac{\beta_K}{2}\geq&\hat{s}_{K+1}^{-1}\max\limits_{w\in\RR^{d+1}}\lbrace 
	\sum_{k=0}^K-\langle 
	\overline{G}_k,w-w_k\rangle-\frac{\beta_K}{2}\norm{w-w_0}^2_2\rbrace\nonumber\\
	=&\hat{s}_{K+1}^{-1}\max\limits_{w\in\RR^{d+1}}\lbrace 
	\sum_{k=0}^K-\langle 
	\frac{G_x(w_k)}{\norm{G(w_k)}_2},x-x_k\rangle+\langle 
	\frac{G_{\lambda}(w_k)}{\norm{G(w_k)}_2},\lambda-\lambda_k\rangle-\frac{\beta_K}{2}\norm{w-w_0}^2_2\rbrace\nonumber\\
	\geq&\hat{s}_{K+1}^{-1}\max\limits_{w\in\RR^{d+1}}\lbrace 
	\sum_{k=0}^K 
	\frac{F(x_k,\lambda_k)-F(x,\lambda_k)}{\norm{G(w_k)}_2}+\frac{F(x_k,\lambda)-F(x_k,\lambda_k)}{\norm{G(w_k)}_2}-\frac{\beta_K}{2}\norm{w-w_0}^2_2\rbrace\nonumber\\
	=&\hat{s}_{K+1}^{-1}\max\limits_{w\in\RR^{d+1}}\lbrace 
	\sum_{k=0}^K 
	\frac{F(x_k,\lambda)-F(x,\lambda_k)}{\norm{G(w_k)}_2}-\frac{\beta_K}{2}\norm{w-w_0}^2_2\rbrace\nonumber\\
	=&\max\limits_{w\in\RR^{d+1}}\lbrace 
	\hat{s}_{K+1}^{-1}\sum_{k=0}^K 
	\frac{F(x_k,\lambda)-F(x,\lambda_k)}{\norm{G(w_k)}_2}-\hat{s}_{K+1}^{-1}\frac{\beta_K}{2}\norm{w-w_0}^2_2\rbrace\nonumber\\
	\geq&\max\limits_{w\in\RR^{d+1}}\lbrace 
	F(\overline{x}_{K+1},\lambda)-F(x,\overline{\lambda}_{K+1})-\hat{s}_{K+1}^{-1}\frac{\beta_K}{2}\norm{w-w_0}^2_2\rbrace\nonumber\\
	=&\max\limits_{w\in\RR^{d+1}}\lbrace f(\overline{x}_{K+1})+\lambda\overline{f}(\overline{x}_{K+1})
	-f(x)-\overline{\lambda}_{K+1}\overline{f}(x)-\hat{s}_{K+1}^{-1}\frac{\beta_K}{2}\norm{w-w_0}^2_2\rbrace,\label{maxx}
	\end{alignat} 
	where the second inequality follows from the convexity and linearity of $F(x,\lambda)$ in $x$ and  
	$\lambda$, respectively. The third inequality uses Jensen's inequality and linearity: 
	$\overline{x}_{K+1}=\hat{s}_{K+1}^{-1}\sum_{k=0}^K\frac{x_K}{\norm{G(w_k)}_2}$ is a convex 
	combination of $\{x_k\}$, so by Jensen's inequality, $$F(\overline{x}_{K+1},\lambda)\leq 
	\hat{s}_{K+1}^{-1}\sum_{k=0}^K \frac{F(x_k,\lambda)}{\norm{G(w_k)}_2},$$ 
	and $\overline{\lambda}_{K+1}=\hat{s}_{K+1}^{-1}\sum_{k=0}^K\frac{\lambda_k}{\norm{G(w_k)}_2}$,
	so by the linearity of of $F(x,\lambda)$ in $\lambda$, 
	$$\hat{s}_{K+1}^{-1}\sum_{k=0}^K 
	\frac{-F(x,\lambda_k)}{\norm{G(w_k)}_2}=-F(x,\overline{\lambda}_{K+1}).$$ 
	Given the maximum function, the inequality \eqref{maxx} holds for any choice of $w$. We consider 
	two cases, the first being $x=\overline{x}_{K+1}$ and 
	$\lambda=\overline{\lambda}_{K+1}+1$. From \eqref{maxx},
	\begin{alignat}{6}
	\hat{s}_{K+1}^{-1}\frac{\beta_K}{2}\geq&f(\overline{x}_{K+1})+(\overline{\lambda}_{K+1}+1)\overline{f}(\overline{x}_{K+1})
	-f(\overline{x}_{K+1})-\overline{\lambda}_{K+1}\overline{f}(\overline{x}_{K+1})\nonumber\\
	&-\hat{s}_{K+1}^{-1}\frac{\beta_K}{2}\norm{[\overline{x}_{K+1};\overline{\lambda}_{K+1}+1]-w_0}^2_2\nonumber\\
	=&\overline{f}(\overline{x}_{K+1})-\hat{s}_{K+1}^{-1}\frac{\beta_K}{2}\norm{[\overline{x}_{K+1};\overline{\lambda}_{K+1}+1]-w_0}^2_2.\label{c11}
	\end{alignat}	
	Further,
	\begin{alignat}{6}
	\norm{[\overline{x}_{K+1};\overline{\lambda}_{K+1}+1]-w_0}_2\leq&\norm{\overline{w}_{K+1}-w_0}_2+1\nonumber\\
	=&\norm{\overline{w}_{K+1}-w^*+w^*-w_0}_2+1\nonumber\\
	\leq&\norm{\overline{w}_{K+1}-w^*}_2+\norm{w^*-w_0}_2+1\nonumber\\
	\leq&\hat{s}_{K+1}^{-1}\sum_{k=0}^K\frac{\norm{w_k-w^*}_2}{\norm{G(w_k)}_2}+\norm{w^*-w_0}_2+1\nonumber\\
	\leq&\hat{s}_{K+1}^{-1}\sum_{k=0}^K\frac{\norm{w_0-w^*}_2+1}{\norm{G(w_k)}_2}+\norm{w^*-w_0}_2+1\nonumber\\
	=&2(\norm{w_0-w^*}_2+1),\label{c12}
	\end{alignat} 
	where the third inequality uses Jensen's inequality: 
	$\overline{w}_{K+1}=\hat{s}_{K+1}^{-1}\sum_{k=0}^K\frac{w_k}{\norm{G(w_k)}_2}$ is a convex 
	combination of $\{w_k\}$. The fourth inequality uses Property 
	\ref{bounded}. Combining \eqref{c11} 
	and \eqref{c12}, 
	\begin{alignat}{6}
	\overline{f}(\overline{x}_{K+1})\leq 
	\hat{s}_{K+1}^{-1}\frac{\beta_K}{2}(4(\norm{w_0-w^*}_2+1)^2+1).\nonumber
	\end{alignat}	
	The second case will use $w=w^*$. Starting from \eqref{maxx},
	\begin{alignat}{6}
	\hat{s}_{K+1}^{-1}\frac{\beta_K}{2}\geq&f(\overline{x}_{K+1})+\lambda^*\overline{f}(\overline{x}_{K+1})
	-f(x^*)-\overline{\lambda}_{K+1}\overline{f}(x^*)-\hat{s}_{K+1}^{-1}\frac{\beta_K}{2}\norm{w^*-w_0}^2_2\nonumber\\
	\geq&f(\overline{x}_{K+1})-f(x^*)-\hat{s}_{K+1}^{-1}\frac{\beta_K}{2}\norm{w^*-w_0}^2_2,\nonumber
	\end{alignat} 	
	since $\overline{f}(x^*)=0$ and $\lambda^*\overline{f}(\overline{x}_{K+1})\geq 0$. Rearranging, 
	\begin{alignat}{6}
	f(\overline{x}_{K+1})-f(x^*)\leq 
	\hat{s}_{K+1}^{-1}\frac{\beta_K}{2}(\norm{w_0-w^*}^2_2+1).\nonumber
	\end{alignat}	
	Using Properties \ref{gradbound} and \ref{betabound}, $\hat{s}_{K+1}^{-1}\frac{\beta_K}{2}$ can be 
	bounded as 
	follows. 
	\begin{alignat}{6}
	\hat{s}_{K+1}^{-1}\frac{\beta_K}{2}\leq&\frac{1}{2}(\sum_{k=0}^K\frac{1}{\norm{G(w_k)}_2})^{-1}(\frac{1}{1+\sqrt{3}}+\sqrt{2k+1})\nonumber\\
	\leq&\frac{1}{2}(\sum_{k=0}^K\frac{1}{C})^{-1}(\frac{1}{1+\sqrt{3}}+\sqrt{2k+1})\nonumber\\
	=&\frac{C}{2(K+1)}(\frac{1}{1+\sqrt{3}}+\sqrt{2k+1}).\nonumber
	\end{alignat}\qed	
\end{proof}
Algorithm \ref{alg:sgd} is an optimal method as its convergence rate of $O(\frac{1}{\sqrt{K}})$ from 
Theorem \ref{convergence} matches the lower complexity bound for minimizing the unconstrained version 
of \eqref{eq:1} as discussed in the introduction. The following corollary establishes the  
$O(\min(\epsilon_1,\epsilon_2)^{-2})$ iteration complexity required to achieve an 
$(\epsilon_1,\epsilon_2)$-optimal solution.
\begin{corollary}
	\label{corol}
	An $(\epsilon_1,\epsilon_2)$ optimal solution is obtained after running Algorithm \ref{alg:sgd} 
	for  
	\begin{alignat}{6}
	K\geq\alpha\max\left(\frac{C_1}{\epsilon_1},\frac{C_2}{\epsilon_2}\right)^2\nonumber
	\end{alignat}
	iterations, where $C_1=C(\norm{w_0-w^*}^2_2+1)$, $C_2=C(4(\norm{w_0-w^*}_2+1)^2+1)$, and 
	$\alpha=\frac{1}{2}(\frac{1}{\sqrt{8}(1+\sqrt{3})}+1)^2$. 
\end{corollary}

\begin{proof} 	
	From Theorem \ref{convergence} for $i=1,2$, we need to compute a lower bound on $K$ which 
	ensures that 
	\begin{alignat}{6}
	\frac{C_i}{2(K+1)}(\frac{1}{1+\sqrt{3}}+\sqrt{2K+1})&\leq \epsilon_i.\nonumber
	\end{alignat}		
	Since for $K\geq 1$,
	\begin{alignat}{6}
	&\frac{\sqrt{K}}{2(K+1)}(\frac{1}{1+\sqrt{3}}+\sqrt{2K+1})\nonumber\\
	=&\frac{\sqrt{K}}{2(1+\sqrt{3})(K+1)}+\frac{\sqrt{2K^2+K}}{2(K+1)}\nonumber\\
	<&\frac{1}{4(1+\sqrt{3})}+\frac{1}{\sqrt{2}}\nonumber\\
	=&\frac{1}{\sqrt{2}}(\frac{1}{\sqrt{8}(1+\sqrt{3})}+1)\nonumber\\
	=&\sqrt{\alpha},\nonumber
	\end{alignat}
	it holds that
	\begin{alignat}{6}  	
	\frac{C_i}{2(K+1)}(\frac{1}{1+\sqrt{3}}+\sqrt{2K+1}) \leq 
	\frac{\sqrt{\alpha}C_i}{\sqrt{K}}.\nonumber
	\end{alignat}	
	To ensure convergence within $\epsilon_i$, it is sufficient for 	
	$\epsilon_i\geq\frac{\sqrt{\alpha} C_i}{\sqrt{K}}$, or that
	$K\geq\alpha(\frac{C_i}{\epsilon_i})^2$. Taking the maximum over $i$ gives the result.\qed	
\end{proof}

\section{Conclusion}
In this paper we have established the existence of a simple first-order method for the general convex 
constrained optimization problem \eqref{eq:1} without the need for differentiability nor Lipschitz 
continuity. We see this as a general use algorithm for practitioners since it requires minimal 
knowledge of the problem, with no parameter tuning for its implementation, while still achieving the 
optimal convergence rate for first-order methods.


%
%

\bibliographystyle{spmpsci}      
\bibliography{dualsubgrad}   


\end{document}